\documentclass{article}

\usepackage{float}
\usepackage{color}
 \catcode`@=11
\usepackage{color}
\usepackage{amsthm,amssymb,amsfonts}
\usepackage{amsmath}
\usepackage{graphicx}
\usepackage{epstopdf}
\catcode`@=11 \@addtoreset{equation}{section}

\catcode`@=12

\newtheorem{theorem}{Theorem}[section]

\newtheorem{lemma}{Lemma}[section]

\theoremstyle{definition}

\usepackage[colorlinks=true,urlcolor=red,linkcolor=blue,citecolor=red,bookmarks=red]{hyperref}
\title{Theoretical and numerical indirect  stabilization of coupled wave equations with a single time-delayed damping}
\author{
	{\sc Alhabib Moumni}\\
	Moulay Ismail University of Meknes,\\
	FST Errachidia, MAIS Laboratory, MAMCS Group,\\
	P.O. Box 509, Boutalamine 52000, Errachidia, Morocco\\
	email: alhabibmoumni@gmail.com\\
		{\sc Mohamed Mehdaoui}
	\\
	Moulay Ismail University of Meknes,\\
FST Errachidia, MAIS Laboratory, MAMCS Group,\\
P.O. Box 509, Boutalamine 52000, Errachidia, Morocco\\
email: m.mehdaoui@edu.umi.ac.ma\\
	{\sc Jawad Salhi}\\
	Moulay Ismail University of Meknes,\\
	FST Errachidia, MAIS Laboratory, MAMCS Group,\\
	P.O. Box 509, Boutalamine 52000, Errachidia, Morocco\\
	email: j.salhi@umi.ac.ma\\
{\sc Mouhcine Tilioua}\\
Moulay Ismail University of Meknes,\\
FST Errachidia, MAIS Laboratory, MAMCS Group,\\
P.O. Box 509, Boutalamine 52000, Errachidia, Morocco\\
email: m.tilioua@umi.ac.ma}

\date{}
\begin{document}
	
	\maketitle
	\begin{abstract}
		The focal point of this paper is to theoretically investigate and numerically validate the effect of time delay on the exponential stabilization of a class of coupled hyperbolic systems with delayed and non-delayed dampings. The class in question consists of two strongly coupled wave equations featuring a delayed and non-delayed damping terms on the first wave equation. Through a standard change of variables and semi-group theory, we address the well-posedness of the considered coupled system. Thereon, based on some observability inequalities, we derive sufficient conditions guaranteeing the exponential decay of a suitable energy. On the other hand, from the numerical point of view, we validate the theoretical results in $1D$ domains based on a suitable numerical approximation obtained through the Finite Difference Method. More precisely, we construct a discrete numerical scheme which preserves the energy decay property of its  continuous counterpart. Our theoretical analysis and implementation of our developed numerical scheme assert the effect of the time-delayed damping on the exponential stability of strongly coupled wave equations.
	\end{abstract}
	
	
	\section{Introduction}\label{intro}
	
	Stabilization problems for hyperbolic equations have drawn the attention of several researchers over decades, focusing on finding a stabilizing control input that induces decay in the system's energy over time. Classical research in the context of scalar wave equations has laid the groundwork for understanding such stabilization mechanisms, as highlighted by \cite{Bardos, Nicaise, chen1979control, tebou1998stabilization, d2001weighted, komornik1994decay, liu1999decay} and the references therein. Unlike their scalar counterpart, coupled wave equations offer a more delicate representation by capturing the interdependence among the multiple components of the state of the dynamics, making them crucial in real-world applications. From the physical point of view, such interdependence can be captured by weak coupling involving displacements, and strong coupling involving velocities. 
	
	Research into the stabilization of coupled hyperbolic systems has expanded with substantial contributions \cite{Alabau2013, Alabau2003, Alabau2002, Alabau-Leautaud2013,Avdonin, Bennour, Gerbi2021, Koumaiha2021, Mokhtari, LiuRao2009, Wehbe2011}. These pioneering contributions have achieved, under different settings, that the stabilization of coupled wave equations can be theoretically achieved by controlling only one component of the state. However, they were all based on the simplifying assumption of instantaneous control effects which, in some cases, does not align with practical scenarios where time delays are inevitable due to slower communication and device responses. That being said, the incorporation of time delays into the stabilization framework has proven to be crucial for practical applications, leading to the introduction of delay terms in the control input. This has been investigated in scalar wave equations by several authors \cite{xu2006stabilization, pignotti2012note, nicaise2006stability, nicaise2010stabilization, ammari2010feedback, adimy2003global, enrike1990exponential, ammari2023note, ammari2022well}. However, addressing time delays in coupled wave equations remains less explored. Existing literature for weakly coupled systems includes theoretical studies which can be found in \cite{oliveira2020stability, xu2020uniform, rebiai2014exponential}, while for strongly coupled systems, recent theoretical advancements can be found in \cite{akil2020stability, akil2021stability, silga2022indirect}. Indeed, in \cite{akil2020stability}, the authors employed a frequency domain approach revealing a polynomial energy decay rate, while in \cite{akil2021stability}, they investigated stabilization with local viscoelastic damping. On the other hand, in \cite{silga2022indirect} the authors studied the impact of distributed time delays on the boundary stabilization of multi-dimensional systems. As far as the numerical analysis of the theoretical stabilization results of strongly coupled wave equations is concerned, to the best of our knowledge, the only existing work can be found in \cite{gerbi2021numerical}. As a matter of fact, the authors established a numerical scheme preserving the energy dissipating property for locally strongly coupled wave equations with a single non-delayed damping. Hence, our second main objective is to extend the approach used in \cite{gerbi2021numerical} to the case of locally strongly coupled wave equations with delayed and non-delayed dampings. 

	Taking the preceding discussion into account, this paper contributes to both theoretical and numerical aspects of the stabilization of strongly coupled wave equations by first deriving conditions for exponential energy decay from a theoretical perspective and then validating these findings numerically through a finite difference numerical scheme designed to preserve the energy decay properties of the continuous system. This ensures that the discrete model accurately reflects the decaying behavior observed theoretically. The continuous system in question is governed by the following two coupled wave equations: 
	\begin{equation}\label{coupledstabilityproblem}
	\begin{cases}
	u_{tt}(x,t)-\Delta u(x,t)+b(x) y_{t}(x,t)+a(x)\left[\mu_1 u_t(x,t)+\mu_2 u_t(x,t-\tau)\right]=0, & \text{in } \Omega \times (0, \infty), \\ 
	y_{tt}(x,t)-\Delta y(x,t)-b(x) u_{t}(x,t)=0, & \text{in }\Omega \times (0, \infty), \\ 
	u(x,t)=y(x,t)=0, & \text {on }\partial\Omega \times (0, \infty),\\
	u(x,0)=u_{0}(x), \quad u_t(x,0)=u_1(x), & \text{in } \Omega, \\ 
	y(x,0)=y_{0}(x), \quad y_t(x,0)=y_1(x),  & \text{in } \Omega,\\
	u_t(x,t)=f_0(x,t), &  \text{in } \Omega \times (-\tau,0),
	\end{cases}
	\end{equation}
	wherein $\Omega$ is an open bounded domain of $\mathbb{R}^n \;(n \geq 1)$, with a boundary $\partial \Omega$ of class $C^2$. Moreover,  $a \in L^{\infty}(\Omega)$ is a function such that
	$$
	a(x) \geq 0 \text { a. e. in } \Omega,
	$$
	and
	$$
	a(x)>a_0>0 \quad \text { a. e. in } \omega,
	$$
	where $\omega$  is an open nonempty subset of $\Omega.$

	\noindent
	 On the other hand $(u_0, u_1, y_0, y_1,f_0)$ denotes the initial state which belongs to a suitable space which will be precised in the sequel. Additionally, $b\in W^{1,\infty}(\Omega)$ represents the coupling parameter which is assumed to satisfy the following condition: 
\begin{equation}
\label{conditionb}
\overline{\omega_b}\subset\left\{x\in\Omega:\quad b(x)\neq0\right\},
	\end{equation}
for a non empty open $\omega_b\subset\Omega$. Finally, $\mu_1, \mu_2$ are real positive numbers, while $\tau>0$ stands for the time delay.

In order to put System \eqref{coupledstabilityproblem} into perspective with earlier contributions, let us mention that in the absence of delay, that is, $\mu_2=0$, Gerbi et al. \cite{Gerbi2021,gerbi2021numerical} investigated the stabilization of locally coupled wave equations. More precisely, by employing a frequency domain approach along with a multiplier technique, they established the system's exponential stability under the condition that the coupling region is a subset of the damping region and satisfies the geometric control condition (GCC) as defined in \cite[Definition 1]{Gerbi2021} or \cite{liu1997locally}. For the reader's convenience, here, we briefly recall that a subset $\omega$ of $\Omega$ satisfies the GCC if every ray of the geometrical optics originating from any point $x \in \Omega$ at $t=0$ enters the region $\omega$ within a finite time $T>0$. On the other hand, in the case of a single wave equation, the exponential decay result was derived by Nicaise and  Pignotti in \cite{nicaise2006stability}. More precisely, by considering the following scalar wave equation 
\begin{equation}
	\label{singleproblem}\begin{cases}
		u_{t t}(x, t)-\Delta u(x, t)+a(x)\left[\mu_1 u_t(x,t)+\mu_2 u_t(x,t-\tau)\right]=0, & \text{in } \Omega \times (0, \infty), \\ 
		u(x,t)=0, & \text {on }\partial\Omega \times (0, \infty),\\
		u(x,0)=u_{0}(x), \quad u_t(x,0)=u_1(x), & \text{in } \Omega, \\ 
		u_t(x,t)=f_0(x,t), &  \text{in } \Omega \times (-\tau,0),
	\end{cases}
\end{equation}
featuring a local damping applied to a single equation, the authors established that the exponential stability can be maintained for the whole system. 

Thus, the theoretical contributions of our work are to the extension of the early theoretical results presented in \cite{Gerbi2021,nicaise2006stability} to the case given by System \eqref{coupledstabilityproblem}. To this end, in line with the earlier contributions \cite{Gerbi2021,gerbi2021numerical}, in this paper, we assume that $\omega$ satisfies the geometric control condition GCC, furthermore we consider the following assumption 
	\begin{equation}
	\label{conditionbdamping}
	\omega_b\subset \omega.
	\end{equation}
Additionally, as System \eqref{coupledstabilityproblem} shows, we assume that only one of the two components of the unknown is under feedback, in order to investigate whether it is possible to stabilize the complete vector state solution to the coupled hyperbolic System \eqref{coupledstabilityproblem}, for a coupling function $b$ satisfying \eqref{conditionb}, by using a single damping. As far as the numerical contribution of our work in concerned, we aim to extend the approach developed in \cite{gerbi2021numerical}, in the absence of time-delayed damping, to the case of the coupled hyperbolic System \eqref{coupledstabilityproblem} where such time-delayed damping is present.
	
The rest of this paper is organized as follows. In Section \ref{section2}, we deal with the well-posedness  of Problem \eqref{coupledstabilityproblem} by means of semigroup theory. In Section \ref{section3}, we explore the exponential decay properties of Problem \eqref{coupledstabilityproblem} under some conditions correlating  $\mu_1$ and $\mu_2$. In Section \ref{section4}, we formulate a numerical approximation of Problem \eqref{coupledstabilityproblem} based on a finite difference discretization scheme and prove that it preserves the energy decaying property. Finally, in Section \ref{section5}, we validate the theoretical results by performing various numerical simulations for different damping parameters. 
	
In what follows, $C$ will signify a generic positive constant, whose value may differ across various situations.
	\section{Well-posedness}\label{section2}
	Before addressing the questions related to stability  issues, we first deal with the well-posedness of System \eqref{coupledstabilityproblem}. 
	As in \cite{nicaise2006stability}, let us introduce the following function:
	$$
	z(x, \rho, t)=u_t(x, t-\tau \rho), \quad x \in \Omega,\text{ } \rho \in(0,1),\text{ } t>0.
	$$
	Then, Problem \eqref{coupledstabilityproblem} is equivalent to
	\begin{equation}\label{problemwithoutdelay}
	\begin{cases}
	u_{tt}(x,t)-\Delta u(x,t)+b y_{t}(x,t)+a(x)\left[\mu_1u_t(x,t)+\mu_2 z(x,1,t)\right]=0, & \text{in } \Omega \times (0, \infty), \\ 
	y_{tt}(x,t)-\Delta y(x,t)-b u_{t}(x,t)=0, & \text{in }\Omega \times (0, \infty), \\
	\tau z_{t}(x,\rho,t)+z_{\rho}(x,\rho,t)=0,  & \text{in }\Omega \times (0,1)\times  (0, \infty), \\ 
	u(x,t)=y(x,t)=0, & \text {on }\partial\Omega \times (0, \infty),\\
	z(x,0,t)=u_t(x,t), & \text{in }\Omega \times (0, \infty), \\
	u(x,0)=u_{0}(x), \quad u_t(x,0)=u_1(x), & \text{in } \Omega, \\ 
	y(x,0)=y_{0}(x), \quad y_t(x,0)=y_1(x),  & \text{in } \Omega,\\
	z(x,\rho,0)=f_0(x,-\rho \tau), &\text{in } \Omega \times (0,1).
	\end{cases}
	\end{equation}
We	consider the following energy space associated to System \eqref{problemwithoutdelay}:
\begin{equation}\label{energyspace}
\mathcal{H}=\left(H_{0}^{1}(\Omega)\times L^2(\Omega)\right)^2\times L^2(\Omega \times (0, 1)),
\end{equation}
with the scalar product
$$
\langle U, \widetilde{U}\rangle_{\mathcal{H}}=\int_{\Omega}( \nabla u_{1}\nabla \widetilde{u}_{1}+u_{2} \widetilde{u}_{2}+\nabla u_{3} \nabla\widetilde{u}_{3}+u_{4} \widetilde{u}_{4})\,dx+\xi\int_{\Omega}\int_{0}^{1}z(x,\rho)\widetilde{z}(x,\rho),
$$
for all $U=(u_{1}, u_{2}, u_{3}, u_{4},z), \widetilde{U}=(\widetilde{u}_{1}, \widetilde{u}_{2}, \widetilde{u}_{3}, \widetilde{u}_{4},\widetilde{z}) \in \mathcal{H},$ where $\xi$ is a positive number.

Next, we define the linear unbounded operator $\mathcal{A}: D(\mathcal{A})\subset \mathcal{H} \rightarrow \mathcal{H}$ by
\begin{equation}\label{theoperator}
\begin{cases}
D(\mathcal{A})&=\left\{(u_{1}, u_{2}, u_{3}, u_{4},z)\in ((H^2(\Omega)\cap H^1_0(\Omega))\times H^{1}_{0}(\Omega))^2\times L^2(\Omega ;H^1(0,1)),\right.\\
&\left. \quad z(.,0)=u_2(.) \text{ in }\Omega\right\},\\
\mathcal{A} U&=(u_{2},\Delta u_{1}-b u_{4}-a\mu_1 u_2-a\mu_2z(.,1), u_{4}, \Delta u_{3}+b u_{2},-\tau^{-1}z_{\rho}), \\
&\quad \forall U=(u_{1}, u_{2}, u_{3}, u_{4},z) \in D(\mathcal{A}).
\end{cases}
\end{equation}
Thus, System \eqref{problemwithoutdelay} can be rewritten as an abstract Cauchy problem. Indeed, setting $U(t)=(u(t),u_t(t),y(t),y_t(t),z(t))$, we obtain
\begin{equation}\label{cauchyproblem}
\begin{cases}
U^\prime(t)&= \mathcal{A} U(t)\quad t\geq0,\\
U(0)&=U_0,
\end{cases}
\end{equation}
where $U(0)=(u_0,u_1,y_0,y_1,f_0(.,-.\tau))$.

The existence and uniqueness result reads as follows.
\begin{theorem}\label{wellposed}
	For any $U_{0}=\left(u_{0}, u_{1}, y_{0}, y_{1},f_0\right) \in \mathcal{H}$, there exists a unique solution $U=(u,u_t,y,y_t,z)\in C^{0}([0,\infty); \mathcal{H})$ of System \eqref{cauchyproblem}. Moreover, if $U_{0}=\left(u_{0}, u_{1}, y_{0}, y_{1},f_0\right) \in D(\mathcal{A})$, then
	$$
	U \in C^{0}([0,\infty); D(\mathcal{A})) \cap C^{1}([0,\infty); \mathcal{H}).
	$$
\end{theorem}
\begin{proof}
		We will prove that the operator $\mathcal{A}$ defined in \eqref{theoperator} generates a strongly continuous semi-group of contractions on the Hilbert
	space $\mathcal{H}$ defined by \eqref{energyspace}. As in \cite{pignotti2012note} and 
	according to \cite[Proposition 2.2.6]{Cazenave1998}, it is sufficient to prove that there exists a positive constant $c$ such that  the unbounded operator $\mathcal{A}-cI$ is dissipative on $\mathcal{H}$  and that  $\lambda I- \mathcal{A}$ is surjective for some $\lambda>0$, where
	$$
	I= \operatorname{diag}(Id,Id,Id,Id,Id),
	$$
	where $Id$ denotes the identity operator.

	\underline{$\mathcal{A}-cI$ is dissipative for some $c>0$.} Let $U=(u_1,u_2,u_3,u_4,z)\in D(\mathcal{A})$. Then, by integrating by parts, we obtain
	\begin{align*}
	\langle \mathcal{A}U,U\rangle_{\mathcal{H}} = &\langle \big(u_{2},\Delta u_{1}-b u_{4}-a\mu_1u_2-a\mu_2z(.,1), u_{4}, \Delta u_{3}+b u_{2},-\tau^{-1}z_{\rho}\big),\big(u_1,u_2,u_3,u_4,z\big) \rangle_{\mathcal{H}}\\
	=&\int_{\Omega}\Big(\nabla u_1\nabla u_2+(\Delta u_{1}-b u_{4}-a\mu_1u_2-a\mu_2z(.,1))u_2+\nabla u_3\nabla u_4+(\Delta u_{3}+b u_{2})u_4\Big) dx\\
	&-\xi\tau^{-1}\int_{\Omega}\int_{0}^{1}z(x,\rho)z_{\rho}(x,\rho)\,dx \,d\rho\\
	=&-\mu_1\int_{\Omega}a u_2^2(x)\,dx-\mu_2\int_{\Omega}az(x,1)u_2(x)\,dx+\frac{\xi\tau^{-1}}{2}\int_{\Omega}z^2(x,0)\,dx-\frac{\xi\tau^{-1}}{2}\int_{\Omega}z^2(x,1)\,dx\\
	=&-\mu_1\int_{\Omega}a u_2^2(x)\,dx-\mu_2\int_{\Omega}az(x,1)u_2(x)\,dx+\frac{\xi\tau^{-1}}{2}\int_{\Omega}u_2^2(x)\,dx-\frac{\xi\tau^{-1}}{2}\int_{\Omega}z^2(x,1)\,dx\\
	\leq &\left(\frac{\xi\tau^{-1}}{2}+\frac{\mu_2^2\|a\|^2_{L^{\infty}(\Omega)}}{2\xi\tau^{-1}}\right)\int_{\Omega}u_2^2(x)\,dx.	
	\end{align*}
Taking $c=\frac{\xi\tau^{-1}}{2}+\frac{\mu_2^2\|a\|^2_{L^{\infty}(\Omega)}}{2\xi\tau^{-1}},$ we get that 
	$$\langle (\mathcal{A}-cI)U,U\rangle_{\mathcal{H}}\leq0.$$
	Therefore the operator $\mathcal{A}-cI$ is dissipative.
	
	\underline{$\lambda I-\mathcal{A}$ is surjective for any $\lambda>0.$} Given $F=\left(f_{1}, f_{2}, f_{3}, f_{4},h\right) \in \mathcal{H}$,
	we seek  $U=\left(u_{1}, u_{2}, u_{3}, u_{4},z\right) \in D(\mathcal{A})$ such that
	\begin{equation}\label{probdisspdege}
	\lambda U-\mathcal{A} U=F \Leftrightarrow
	\begin{cases}
	\lambda u_{1}-u_{2} &=f_{1}, \\
	\lambda u_{2}-\Delta u_1+b u_{4}+a(\mu_1u_2+\mu_2z(.,1)) &=f_{2}, \\
\lambda 	u_{3}-u_{4} &=f_{3}, \\
\lambda 	u_{4}-\Delta u_3-b u_{2} &=f_{4},\\
\lambda z+\tau^{-1} z_{\rho}&=h.
	\end{cases}
	\end{equation}
	Analogously to \cite{pignotti2012note}, we suppose that we have found $u_1$  with the right regularity. Then, we set
	\begin{equation}
	\label{u2byu1}
	u_2=\lambda u_1-f_1
	\end{equation}
	and aim to determine $z$.
	From the definition of $\mathcal{D}(\mathcal{A})$, we aim to find $z$ such that
$$
	z(x, 0)=u_2(x), \text { for } x \in \Omega;
	$$
	and from \eqref{probdisspdege},
	$$
	\lambda z(x, \rho)+\tau^{-1} z_\rho(x, \rho)=h(x, \rho), \text { for } x \in \Omega,\text{ } \rho \in(0,1).
	$$
	
	Consequently,  we obtain
	$$
	z(x, \rho)=e^{-\lambda \rho \tau} u_2(x)+\tau e^{-\lambda \rho \tau} \int_0^\rho h(x, \sigma) e^{\lambda \sigma \tau} d \sigma.
	$$
	
	Using \eqref{u2byu1} we get
	\begin{equation}
	\label{zdelay}
	z(x, \rho)=\lambda u_1(x) e^{-\lambda \rho \tau}-f_1(x) e^{-\lambda \rho \tau}+\tau e^{-\lambda \rho \tau} \int_0^\rho h(x, \sigma) e^{\lambda \sigma \tau} d \sigma
	\end{equation}
	and in particular
	$$
	z(x, 1)=\lambda u_1(x) e^{-\lambda \tau}+z_0(x)
	$$
	where $z_0 \in L^2\left(\Omega\right)$ is defined by
	$$
	z_0(x)=-f_1(x) e^{-\lambda \rho \tau}+\tau e^{-\lambda \rho \tau} \int_0^1 h(x, \sigma) e^{\lambda \sigma \tau} d \sigma.
	$$
	Eliminating $u_{2}$ and $u_{4}$ in \eqref{probdisspdege}, we get the following system
	\begin{equation}
	\label{prbfidege}
	\begin{cases}
	\lambda^2u_{1}-\Delta u_1+b\lambda u_{3}+a\mu_1\lambda u_1+a\mu_2\lambda e^{-\lambda\tau}&=\lambda f_{1}+f_{2}+b f_{3}+a\mu_1 f_1-a\mu_2z_0, \\
	\lambda^2u_{3}-\Delta u_3-b\lambda u_{1}&=\lambda f_{3}-b f_{1}+f_{4}.
	\end{cases}
	\end{equation}
	Now, we consider the bilinear form $\Gamma:\left(H_{0}^{1}(\Omega)\times H_{0}^{1}(\Omega)\right)^{2} \rightarrow \mathbb{R}$ given by
	$$
	\begin{aligned}
	\Gamma\left(\left(u_{1}, u_{3}\right),\left(\phi_{1}, \phi_{2}\right)\right)=&\int_{\Omega}(\lambda^2u_{1} \phi_{1}+\nabla u_{1} \nabla\phi_{1} +\lambda^2u_{3} \phi_{2}+\nabla u_{3} \nabla\phi_{2})d x+\mu_1\lambda\int_{\Omega}au_1\phi_1\,dx\\
	&+\mu_2\lambda e^{-\lambda\tau}\int_{\Omega}u_1\phi_1\,dx+b\lambda\int_{\Omega} \left(u_{3} \phi_{1}-u_{1} \phi_{2}\right)\,dx
	\end{aligned}
	$$
	and the linear form $L:H_{0}^{1}(\Omega)\times H_{0}^{1}(\Omega)\rightarrow \mathbb{R}$ given by
	$$
	L\left(\phi_{1}, \phi_{2}\right)=\int_{\Omega}\left(\lambda f_{1}+f_{2}+b f_{3}+a\mu_1 f_1-a\mu_2z_0\right) \phi_{1}\,dx + \int_{\Omega}\left(\lambda f_{3}-b f_{1}+f_{4}\right) \phi_{2}\,dx.
	$$
	In view of Poincaré inequality, $\Gamma$ is a continuous bilinear form on $H_{0}^{1}(\Omega)\times H_{0}^{1}(\Omega)$ and $L$ is a continuous linear functional on $H_{0}^{1}(\Omega)\times H_{0}^{1}(\Omega)$. Moreover, it is easy to check that $\Gamma$ is also coercive on $H_{0}^{1}(\Omega)\times H_{0}^{1}(\Omega)$ for any $\lambda>0.$ Indeed, let $(u_1,u_3)\in H_{0}^{1}(\Omega)\times H_{0}^{1}(\Omega)$ 
	$$\begin{aligned}
\Gamma((u_1,u_3),(u_1,u_3))=&\int_{\Omega}\left(\lambda^2u_1^2+|\nabla u_1|^2+\lambda^2u_3^2+|\nabla u_3|^2\right)\,dx+\mu_1\lambda \int_{\Omega}au_1^2\,dx+\mu_2\lambda e^{-\lambda \tau}\int_{\Omega}au_1^2\,dx\\
	\geq& \int_{\Omega}\left(\lambda^2u_1^2+|\nabla u_1|^2+\lambda^2u_3^2+|\nabla u_3|^2\right)\,dx\\
	&\geq \|(u_1,u_3)\|_{H^{1}_{0}(\Omega)\times H^{1}_{0}(\Omega)}^2.
	\end{aligned}$$
	By applying the Lax–Milgram theorem, we deduce that there exists a unique solution $\left(u_{1}, u_{3}\right) \in H_{0}^{1}(\Omega)\times H_{0}^{1}(\Omega)$ of
	\begin{equation}\label{pv}
	\Gamma\left(\left(u_{1}, u_{3}\right),\left(\phi_{1}, \phi_{2}\right)\right)= L\left(\phi_{1}, \phi_{2}\right)\, \text{for all } (\phi_1,\phi_2)\in H_{0}^{1}(\Omega)\times H_{0}^{1}(\Omega).
	\end{equation}
	 If $(\phi_1,\phi_{2})\in \mathcal{D}(\Omega)\times \mathcal{D}(\Omega)$, then $(u_1,u_3)$ solves \eqref{prbfidege} in $\mathcal{D}^{\prime}(\Omega)\times \mathcal{D}^{\prime}(\Omega)$ and so $(u_1,u_3)\in \left(H^{2}(\Omega)\right)^{2}$.
	 	Now, take $(u_2,u_4):=(\lambda u_1-f_1,\lambda u_3-f_3)$; then $(u_2,u_4)\in \left(H_{0}^{1}(\Omega)\right)^{2}$ and $z$ as in \eqref{zdelay}.
	Thus, we have found $U=\left(u_{1}, u_{2}, u_{3}, u_{4},z\right) \in D(\mathcal{A})$, which satisfies \eqref{probdisspdege}. Consequently $\lambda I-\mathcal{A}$ is surjective, for $\lambda$ positive, and therefore $\lambda I -(\mathcal{A}-cI)$ is also surjective. Conclusively, the
	Lumer–Phillips Theorem implies that $\mathcal{A}-cI$ generates a strongly
	continuous semigroup of contractions in $\mathcal{H}$. Hence, the operator
	$\mathcal{A}$ generates a strongly continuous semigroup of contraction
	in $\mathcal{H}$. Consequently, the well-posedness result follows from the Hille-Yosida theorem (see \cite[Theorem 4.5.1]{barbu} or  \cite[Theorem A.11]{Coron2007}).
\end{proof}
\section{Stability result}
\label{section3}
In this section, we study the exponential stability of System \eqref{coupledstabilityproblem}.
To this aim we first define its corresponding energy by
\begin{equation}\label{energysystem}
\begin{aligned}
E(t)&:=\frac{1}{2} \left[\int_{\Omega}\left\{u_{t}^{2}(x,t)+|\nabla u|^2(x,t)+y_{t}^{2}(x,t)+|\nabla y|^2(x,t)\right\}\,dx\right]+\frac{\xi\tau^{-1}}{2}\int_{\Omega}a(x)\int_{t-\tau}^{t}u_t^2(x,s)dsdx, \\
& \forall t \geq 0.
\end{aligned}
\end{equation}
In order to prove the exponential stability result, we need to impose a condition correlating the coefficients $\mu_1$ and $\mu_2$ which correspond, respectively, to the non-delayed and delayed dampings. The condition that we shall impose is the following:
\begin{equation}
\label{conditionmu}
\mu_2<\mu_1.
\end{equation}
We point out that this condition aligns with the one obtained in \cite{nicaise2006stability} to assure the exponential stability in the case of a single wave equation. In the opposed case  ($\mu_2\geq\mu_1$),  the exponential stability no longer holds. Moreover, as in the case of a single wave equation, we choose the constant $\xi$ such that the following condition is satisfied 
\begin{equation}
\label{conditionxi}
\tau \mu_2<\xi<\tau\left(2 \mu_1-\mu_2\right).
\end{equation}

Then, under Condition \eqref{conditionmu} and by choosing $\xi$ satisfying Condition \eqref{conditionxi}, the energy defined in \eqref{energysystem}  exhibits a decay with respect to time. More precisely, we have the following result.
\begin{lemma}
	Let $U=(u, u_t, y, y_t)$ be a regular solution to System \eqref{coupledstabilityproblem}. Then, the energy $E$ associated to $(u,y)$ satisfies
	\begin{equation}\label{decergi}
	\begin{aligned}
E^{\prime}(t)&\leq \left(-\mu_1+\frac{\mu_2}{2}+\frac{\xi\tau^{-1}}{2}\right) \int_{\Omega}a(x)u_t^2(x,t)\,dx+\left(\frac{\mu_2}{2}-\frac{\xi\tau^{-1}}{2}\right)\int_{\Omega}a(x)u_t^2(x,t-\tau)\,dx. \\
	&\quad \forall t \geq 0.
	\end{aligned}
	\end{equation}	
	Consequently, if \eqref{conditionmu} holds and $\xi$ satisfies \eqref{conditionxi}, the energy $E(.)$ is nonincreasing.
\end{lemma}
\begin{proof}
	By differentiating the energy given by \eqref{energysystem}, we have 
	$$\begin{aligned}
	E^{\prime}(t)=&\int_{\Omega}\left\{u_t(x,t) u_{tt}(x,t)+\nabla u(x,t)\nabla u_t(x,t) +y_t(x,t)y_{tt}(x,t)+\nabla y(x,t)\nabla y_t(x,t)\right\}\,dx\\
	&+\frac{\xi\tau^{-1}}{2}\int_{\Omega}a(x)u_t^2(x,t)\,dx-\frac{\xi\tau^{-1}}{2}\int_{\Omega}a(x)u_t^2(x,t-\tau)\,dx\\
	=&\int_{\Omega}\left\{u_t(x,t) \left(\Delta u(x,t)-by_t(x,t)-a(x)(\mu_1u_t(x,t)+\mu_2u_t(x,t-\tau))\right)+\nabla u(x,t)\nabla u_t(x,t)\right.\\
	&\left.+y_t(x,t)\left(\Delta y(x,t)+b u_{t}(x,t)+\nabla y(x,t)\nabla y_t(x,t)\right.\right\}\,dx\\
	&+\frac{\xi\tau^{-1}}{2}\int_{\Omega}a(x)u_t^2(x,t)\,dx-\frac{\xi\tau^{-1}}{2}\int_{\Omega}a(x)u_t^2(x,t-\tau)\,dx\\
	=&-\mu_1\int_{\Omega}a(x)u_t^2(x,t)\,dx-\mu_2
	\int_{\Omega}a(x)u_t(x,t)u_t(x,t-\tau)\,dx+\frac{\xi\tau^{-1}}{2}\int_{\Omega}a(x)u_t^2(x,t)\,dx\\
	&-\frac{\xi\tau^{-1}}{2}\int_{\Omega}a(x)u_t^2(x,t-\tau)\,dx\\
	&\leq \left(-\mu_1+\frac{\mu_2}{2}+\frac{\xi\tau^{-1}}{2}\right) \int_{\Omega}a(x)u_t^2(x,t)\,dx+\left(\frac{\mu_2}{2}-\frac{\xi\tau^{-1}}{2}\right)\int_{\Omega}a(x)u_t^2(x,t-\tau)\,dx.
	\end{aligned}$$
	The two conditions \eqref{conditionmu} and \eqref{conditionxi} guarantee that 
	$$-\mu_1+\frac{\mu_2}{2}+\frac{\xi \tau^{-1}}{2} \leq 0, \quad \frac{\mu_2}{2}-\frac{\xi \tau^{-1}}{2} \leq 0.$$
	Then, the energy $E$ is nonincreasing.
\end{proof}


Now that we have established the conditions under which the energy $E$ is nonincreasing, we proceed to address the question of determining the rates at which its decay occurs. To this end, we consider the following system:
	\begin{equation}\label{problemobsg}
\begin{cases}
w_{tt}-\Delta w+b \theta_t=0, & \text{in } \Omega \times (0, \infty), \\ 
\theta_{tt}-\Delta \theta-b w_{t}=0, & \text{in }\Omega \times (0, \infty), \\ 
w=\theta=0, & \text {on }\partial\Omega \times (0, \infty),\\
w(x,0)=w_{0}(x), \quad w_t(x,0)=w_1(x), & \text{in } \Omega, \\ 
\theta(x,0)=\theta_0(x), \quad \theta_t(x,0)=\theta_1(x),  & \text{in } \Omega.
\end{cases}
\end{equation}
with initial data $\left(w_0, w_1,\theta_0,\theta_1\right) \in \left(H_0^1(\Omega) \times L^2(\Omega)\right)^2$. 

Furthermore, we denote by $E_S(.)$ the standard energy corresponding to strongly coupled  wave equations
\begin{equation}
\label{energyobs}
E_S(t):=\frac{1}{2} \int_{\Omega}\left\{w_t^2(x, t)+|\nabla w(x, t)|^2+\theta_t^2(x,t)+|\nabla \theta(x,t)|^2\right\} d x.
\end{equation}

It is well-known (see \cite{Gerbi2021}) that Problem \eqref{problemobsg} is conservative. Moreover, for any function $b,c\in W^{1,\infty}(\Omega)$ and non empty open sets $\omega_{c^{+}}$ and $\omega_{b,0}$  satisfying 
\begin{equation}
\tag{LH1}\label{LH1}
\overline{\omega_{b,0}}\subset\left\{x\in\Omega: \quad b(x)\neq0\right\}
\end{equation}
\begin{equation}
\tag{LH2}\label{LH2}
\overline{\omega_{c^{+}}}\subset\left\{x\in\Omega: \quad c(x)>0\right\},
\end{equation}
and $$\omega_{b,0}\subset\omega_{c^{+}}$$
 satisfying the GCC,
 the following observability inequality holds if the observability time is sufficiently large.
\begin{theorem}
	\label{Theoremobs}
 There exists a time $T_0>0$ such that for any time $T>T_0$ there exists a positive constant $C$ (depending on $T$ ) for which
$$
E_S(0) \leq C \int_0^T \int_\Omega c(x) w_t^2(x, t) d x d t
$$
for any weak solution to Problem \eqref{problemobsg}.
\end{theorem}
With the aforementioned observability inequality at hand, we can prove the following exponential stability result for the  Problem \eqref{coupledstabilityproblem}.
\begin{theorem}
	\label{Theoremexpstability}
 Assume that Conditions \eqref{conditionb}, \eqref{conditionbdamping}, \eqref{conditionmu} and \eqref{conditionxi} hold and $\omega$ satisfies the GCC. Then, there exist two positive constants $K, \tilde{\mu}$ such that
\begin{equation}
\label{expstability}
E(t) \leq K e^{-\tilde{\mu} t} E(0), \quad t>0,
\end{equation}
for any solution to Problem \eqref{coupledstabilityproblem}.
\end{theorem}
\begin{proof}
	In the spirit of \cite{enrike1990exponential} and \cite{nicaise2006stability}, we can decompose the solution $(u,y)$ of \eqref{coupledstabilityproblem}  as $u=w+\tilde{w}$ and $y=\theta+\tilde{\theta},$ 
	where $(w,\theta)$ is the solution to the problem 
	\begin{equation}\label{problemobs}
	\begin{cases}
	w_{tt}(x,t)-\Delta w(x,t)+b \theta_t(x,t)=0, & \text{in } \Omega \times (0, \infty), \\ 
	\theta_{tt}(x,t)-\Delta \theta (x,t)-b w_{t}(x,t)=0, & \text{in }\Omega \times (0, \infty), \\ 
	w(x,t)=\theta(x,t)=0, & \text {on }\partial\Omega \times (0, \infty),\\
	w(x,0)=u_{0}(x), \quad w_t(x,0)=u_1(x), & \text{in } \Omega, \\ 
	\theta(x,0)=y_0(x), \quad \theta_t(x,0)=y_1(x),  & \text{in } \Omega,
	\end{cases}
	\end{equation}
	and $(\tilde{w},\tilde{\theta})$ solution to the problem
	\begin{equation}\label{problemdecomp}
	\begin{cases}
\tilde{w}_{tt}(x,t)-\Delta \tilde{w}(x,t)+b \tilde{\theta}_t(x,t)=-a\left[\mu_1u_t(x,t)+\mu_2u_t(x,t-\tau)\right], & \text{in } \Omega \times (0, \infty), \\ 
	\tilde{\theta}_{tt}(x,t)-\Delta \tilde{\theta}(x,t)-b \tilde{w}_{t}(x,t)=0, & \text{in }\Omega \times (0, \infty), \\ 
\tilde{w}(x,t)=	\tilde{\theta}(x,t)=0, & \text {on }\partial\Omega \times (0, \infty),\\
	\tilde{w}(x,0)=0, \quad \tilde{w}_t(x,0)=0, & \text{in } \Omega, \\ 
	\tilde{\theta}(x,0)=0, \quad \tilde{\theta}_t(x,0)=0,  & \text{in } \Omega.
	\end{cases}
	\end{equation}
	By \eqref{energysystem} and \eqref{energyobs}, we obtain that
	$$E(0)=E_{S}(0)+\frac{\xi\tau^{-1}}{2}\int_{\Omega}a(x)\int_{-\tau}^{0} u_t^2(x,s)\,ds\,dx=E_{S}(0)+\frac{\xi\tau^{-1}}{2}\int_{\Omega}a(x)\int_{0}^{\tau} u_t^2(x,t-\tau)\,dxdt.$$
	Now, since  $\omega_{b}$ is non empty, then there exist $x_{0}\in\omega_{b}.$ Moreover, since $\omega_{b}$ is an open subset, then there exists $\varepsilon>0$ such that $B(x_{0},2\varepsilon)\subset \omega_{b}\subset\omega.$ Taking $\omega_{b,0}=\omega_{b}\cap B(x_{0},\frac{\varepsilon}{2}),$ it is clear that $\omega_{b,0}$ satisfies Condition \eqref{LH1}. On the other hand  by Urysohn's lemma, there exists a smooth function $0\leq c\leq 1$ such that 
	\begin{equation}
	c(x)=\begin{cases}
1,&\text{ if } \quad x\in B(x_{0},\varepsilon),\\
0,&\text{ if }\quad x\notin B(x_{0},2\varepsilon).
	\end{cases}
	\end{equation} 
Now, we set $\omega_{c^{+}}=B(x_{0},\frac{\varepsilon}{2}).$ Thereby, we can see that $\overline{\omega_{c^{+}}}\subset B(x_{0},\varepsilon)$ and hence the function $c$ satisfies Condition \eqref{LH2} and $\omega_{b,0}\subset\omega_{c^{+}}.$

Thus, we can apply the Theorem \ref{Theoremobs}, to obtain that for every $T>T_{0}$ 
	$$\begin{aligned}
E_{S}(0)&\leq C\int_{0}^{T}\int_{\Omega}c(x)w_t^2(x,t)\,dx\,dt\leq C\int_{0}^{T}\int_{B(x_{0},2\varepsilon)}w_t^2(x,t)\,dx\,dt\leq C\int_{0}^{T}\int_{\omega}w_t^2(x,t)\,dx\,dt \\
&\leq \frac{C}{a_0}\int_{0}^{T}\int_{\omega}a_0w_t^2(x,t)\,dx\,dt\leq C \int_{0}^{T}\int_{\omega}a(x)w_t^2(x,t)\,dx\,dt\leq  C \int_{0}^{T}\int_{\Omega}a(x)w_t^2(x,t)\,dx\,dt.
	\end{aligned}$$
	Hence, we can deduce that  if $T>\max\{T_{0},\tau\}$, it holds that
	\begin{equation}
	\label{majE0}
	\begin{aligned}
E(0) &\leq  C \int_0^T \int_\Omega a(x)w_t^2(x, t) d x d t
	 +\frac{\xi\tau^{-1}}{2}\int_{\Omega}a(x) \int_0^T  u_t^2(x, t-\tau) d t d x\\
	 &\leq  2C \int_\Omega a(x) \int_0^T  \left(u_{t}^2(x,t)+\tilde{w}_t^2(x, t)\right) d t d x
	 +\frac{\xi\tau^{-1}}{2} \int_{\Omega} a(x)\int_0^T  u_t^2(x, t-\tau) d t d x.
	 	\end{aligned}
	 	\end{equation}
	 	Now, observe that from \eqref{problemdecomp}, it follows that
	 	$$\begin{aligned}
	 	\frac{1}{2} \frac{d}{d t} &  \int_{\Omega}\left\{\tilde{w}_t^2(x,t)+|\nabla \tilde{w}(x,t)|^2+\tilde{\theta}_t^2(x,t)+|\nabla \tilde{\theta}(x,t)|^2\right\} d x\\
	 	&=\int_{\Omega}\left\{\tilde{w}_t(x,t)\tilde{w}_{tt}(x,t)+\nabla \tilde{w}(x,t)\nabla \tilde{w}_t(x,t)+\tilde{\theta}_t(x,t)\tilde{\theta}_{tt}(x,t)+\nabla \tilde{\theta}(x,t)\nabla \tilde{\theta}_t(x,t)\right\} \\
	 	& =\int_{\Omega} \tilde{w}_t\left[-\mu_1 a(x)u_t(x,t)-\mu_2 a(x)u_t(x,t-\tau)\right] d x.
	 	\end{aligned}$$
	 	Integrating with respect to time on $[0, t]$, for $t \in(0, T]$, and using the fact that the initial data in \eqref{problemdecomp} is set to zero, we have
	 	$$
	 	\begin{aligned}
	 	\frac{1}{2} \int_{\Omega}\left\{\tilde{w}_t^2(x,t)+|\nabla \tilde{w}(x,t)|^2+\tilde{\theta}_t^2(x,t)+|\nabla \tilde{\theta}(x,t)|^2\right\} &d x =  -\mu_1 \int_{\Omega}a(x)\int_0^t  \tilde{w}_t(x,t) u_t(x,t) d t d x \\
	 	& -\mu_2\int_{\Omega}a(x)\int_0^t  \tilde{w}_t(x,t) u_t(x,t-\tau) d t d x.
	 	\end{aligned}
	 	$$
	 	
	 	By integrating with respect to time, we deduce that
	 	$$
	 	\begin{aligned}
	 	\int_0^T \int_{\Omega} \tilde{w}_t^2(x,t) d x d t \leq& -2 \mu_1 T \int_{\Omega}a(x)\int_0^t  \tilde{w}_t(x,t) u_t(x,t) d x d t\\
	 	 &-2 \mu_2 T\int_{\Omega}a(x)\int_0^t  \tilde{w}_t(x,t) u_t(x,t-\tau) d x d t\\
	 	 \leq &\frac{1}{4}\int_{\Omega}\int_{0}^{T} \tilde{w}_t^2(x,t) d x d t+4T^2\mu_1^2\|a\|_{L^{\infty}(\Omega)}\int_{\Omega}a(x)\int_{0}^{T} u_t^2(x,t) d x d t\\
	 	 &+ \frac{1}{4}\int_{\Omega}\int_{0}^{T} \tilde{w}_t^2(x,t) d x d t+4T^2\mu_2^2\|a\|_{L^{\infty}(\Omega)}\int_{\Omega}a(x)\int_{0}^{T} u_t^2(x,t-\tau) d x d t,
	 	\end{aligned}
	 	$$
	 	which implies that
	 	$$
	 	\begin{aligned}
	 	\int_0^T \int_{\Omega} \tilde{w}_t^2(x,t) d x d t \leq& 8T^2\mu_1^2\|a\|_{L^{\infty}(\Omega)}\int_{\Omega}a(x)\int_{0}^{T} u_t^2(x,t) d x d t\\
	 	&+8T^2\mu_2^2\|a\|_{L^{\infty}(\Omega)}\int_{\Omega}a(x)\int_{0}^{T} u_t^2(x,t-\tau) d x d t.
	 	\end{aligned}
	 	$$ 
	 	We  can conclude that 
	 		$$
	 	\begin{aligned}
	  \int_{\Omega}a(x)\int_0^T \tilde{w}_t^2(x,t) d t d x \leq&	\|a\|_{L^{\infty}(\Omega)}\int_0^T \int_{\Omega} \tilde{w}_t^2(x,t) d x d t\\ \leq& 8T^2\mu_1^2\|a\|_{L^{\infty}(\Omega)}^2\int_{\Omega}a(x)\int_{0}^{T} u_t^2(x,t) d x d t\\
	 	&+8T^2\mu_2^2\|a\|_{L^{\infty}(\Omega)}^2\int_{\Omega}a(x)\int_{0}^{T} u_t^2(x,t-\tau) d x d t.
	 	\end{aligned}
	 	$$ 
	 	This combined with \eqref{majE0} yield that
	 	$$E(0)\leq C_0 \int_{\Omega}a(x)\int_{0}^{T}\left(u_t^2(x,t)+u_t^2(x,t-\tau)\right)\,dt \,dx,$$
	 	where $C_0$ is a positive constant.

	 From \eqref{decergi}, we have
	 	$$
	 	E(T)-E(0) \leq-C \int_{\Omega}a(x) \int_0^T\left\{u_t^2(x, t)+u_t^2(x, t-\tau)\right\} d t dx.
	 	$$
	 	
	 	Hence, we obtain
	 	$$
	 	E(T) \leq E(0) \leq C_0 \int_{\Omega}a(x)\int_{0}^{T}\left(u_t^2(x,t)+u_t^2(x,t-\tau)\right)\,dt \,dx \leq C_0 C^{-1}(E(0)-E(T)),
	 	$$
	 	then
	 	$$
	 	E(T) \leq \tilde{C} E(0)
	 	$$
	 	with $\tilde{C}<1$. 
	 	
	 	This easily implies the stability estimate \eqref{expstability}, since our System \eqref{coupledstabilityproblem} is invariant by translation and the energy $E$ is decreasing.
\end{proof}

\section{Numerical analysis in the case of one dimensional spatial domains}
\label{section4}
When it comes to the numerical validation of the theoretical results pertaining to the stabilization of partial differential equations, one needs to first make sure that the numerical solution, generated by the chosen approximating numerical method, preservers most of the physical properties of its continuous counterpart. In our case, the most important physical property that we shall focus on pertains to the inequality given by \eqref{decergi}, which asserts that, for any given initial state $U_{0}=\left(u_{0}, u_{1}, y_{0}, y_{1},f_0\right) \in \mathcal{H}$, if \eqref{conditionmu} holds and $\xi$ satisfies \eqref{conditionxi}, then the continuous energy of System \eqref{coupledstabilityproblem} is non-increasing.


To proceed further, we consider a one dimensional spatial domain given by $\Omega=(0,1)$. Needless to say, in this case, $\Omega$ is of class $C^2$ and System \eqref{coupledstabilityproblem} reduces to the following one: 

\begin{equation}\label{problem1d}
	\begin{cases}
		u_{tt}(x,t)-u_{xx}(x,t)+b(x) y_{t}(x,t)+a(x)\left[\mu_1 u_t(x,t)+\mu_2 u_t(x,t-\tau)\right]=0, & \text{in } (0,1) \times (0, \infty), \\ 
		y_{tt}(x,t)-y_{xx}(x,t)-b(x) u_{t}(x,t)=0, & \text{in }(0,1) \times (0, \infty), \\ 
		u(0,t)=y(0,t)=0, & \text {on } (0, \infty),\\
		u(1,t)=y(1,t)=0, & \text {on } (0, \infty),\\
		u(x,0)=u_{0}(x), \quad u_t(x,0)=u_1(x), & \text{in } (0,1), \\ 
		y(x,0)=y_{0}(x), \quad y_t(x,0)=y_1(x),  & \text{in } (0,1),\\
		u_t(x,t)=f_0(x,t), &  \text{in } (0,1) \times (-\tau,0).
	\end{cases}
\end{equation}
On the other hand, the energy given by \eqref{energysystem} becomes
\begin{equation}\label{energysystem1d}
	\begin{aligned}
		E(t)&:=\frac{1}{2} \left[\int_{0}^1 \left\{u_{t}^{2}(x,t)+u_x^2(x,t)+y_{t}^{2}(x,t)+| y_x^2(x,t)\right\}\,dx\right]+\frac{\xi\tau^{-1}}{2}\int_{0}^1 a(x)\int_{t-\tau}^{t}u_t^2(x,s)dsdx, \\
		& \forall t \geq 0.
	\end{aligned}
\end{equation}


 The numerical analysis of the stabilization problem of System \eqref{problem1d} features the following two main parts: 
 \begin{enumerate}
 	\item \textbf{Numerical approximation of  the continuous system:}  In this step, we aim to develop by means of the Finite Difference Method a numerical scheme of System \eqref{problem1d} which generates a stable converging discrete solution. In this regard, we point out that  based on the approximations of the operators $\Psi \mapsto \Psi_{tt}$, $\Psi \mapsto \Psi_{xx}$ and $\Psi \mapsto \Psi_x$, various numerical schemes can be constructed. In our case, we choose a three-point approximation of the operator $\Psi \mapsto \Psi_{xx}$ at the current time step,  a three-point approximation of the operator $\Psi \mapsto \Psi_{tt}$ and lastly, a centered difference approximation of the operator $\Psi \mapsto \Psi_x$ at the next time step. Let us point out that such choices guarantee that the discrete solution is stable and converging to the continuous smooth solution corresponding to System \eqref{problem1d}. This can be achieved by classical arguments based on the standard Von Neumann stability analysis for coupled partial differential equations.
 	\item \textbf{Numerical approximation of the continuous energy:} Contrarily to the previous part, the numerical approximation of the integrals featuring in the definition of the energy must be carefully selected in order to assure that the generated discrete energy preserves the decay property given by \eqref{decergi}. In our case, we begin by approximating the operator $\Psi \mapsto \Psi_t^2$ showing in the non-delayed integrands of the energy by a forward approximation. Then, the integral with respect to the spatial variable is approximated by a first-order forward approximation. On the other hand, we use a modified approximation of the operator $\Psi \mapsto \Psi_x^2$ featuring both current and next time steps and then use a first-order forward approximation to approximate the  integral with respect to the spatial variable. As for the last term featuring time delay, we  use  a second-ordered centered difference approximation of the operator $\Psi \mapsto \Psi_t^2$, then, as always, we use  a forward approximation to approximate  the  integral with respect to time and space.  As we will explore in the upcoming subsections, the justification of the aforementioned choices of approximations is purely computational in order to assure that the decaying property of the discrete energy is satisfied.
 \end{enumerate}


\subsection{Numerical approximation of the continuous system}
Let $N$ be a non negative integer. Consider the subdivision of $[0,1]$ given by
$$
0=x_0<x_1<\ldots<x_N<x_{N+1}=1, \quad \text { i.e. } x_j=j \Delta x, j=0, \ldots, N+1 .
$$

Set $t^{n+1}-t^n=\Delta t$ for all $n \in \mathbb{N}$.  For $j=0, \ldots, N+1$, we denote $b_j=b\left(x_j\right), a_j=a\left(x_j\right),$ $u_j^{n}=u(x_j,t^{n})$ and $y_j^{n}=y(x_j,t^{n})$.
Let $\tau=I \Delta t,$ then $t^{n}-\tau=(n-I)\Delta t,$ then we can see that for $n\leq I $ the delayed  term in \eqref{problem1d} is defined by the initial function $f_0.$   The explicit finite-difference discretization of System \eqref{problem1d} is thus, for $n =0,1,\ldots,I$ and $j=1, \ldots, N$:
\begin{equation}
\label{prbdescrr}
\begin{cases}\frac{u_j^{n+1}-2 u_j^n+u_j^{n-1}}{\Delta t^2}- \frac{u_{j+1}^n-2 u_j^n+u_{j-1}^n}{\Delta x^2}+b_j \frac{y_j^{n+1}-y_j^{n-1}}{2 \Delta t}+a_j\mu_1 \frac{u_j^{n+1}-u_j^{n-1}}{2 \Delta t}+a_j\mu_2f_0(x_j,(n-I)\Delta t)=0, \\ \frac{y_j^{n+1}-2 y_j^n+y_j^{n-1}}{\Delta t^2}-\frac{y_{j+1}^n-2 y_j^n+y_{j-1}^n}{\Delta x^2}-b_j \frac{u_j^{n+1}-u_j^{n-1}}{2 \Delta t}=0, \\ u_0^n=u_{N+1}^n=0, \\ y_0^n=y_{N+1}^n=0. & \end{cases}
\end{equation} 
For $n>I,$ the explicit finite-difference discretization of System \eqref{problem1d} reads
\begin{equation}
\label{prbdescr}
\begin{cases}\frac{u_j^{n+1}-2 u_j^n+u_j^{n-1}}{\Delta t^2}- \frac{u_{j+1}^n-2 u_j^n+u_{j-1}^n}{\Delta x^2}+b_j \frac{y_j^{n+1}-y_j^{n-1}}{2 \Delta t}+a_j\mu_1 \frac{u_j^{n+1}-u_j^{n-1}}{2 \Delta t}+a_j\mu_2\frac{u_j^{n-I+1}-u_j^{n-I-1}}{2 \Delta t}=0, \\ \frac{y_j^{n+1}-2 y_j^n+y_j^{n-1}}{\Delta t^2}-\frac{y_{j+1}^n-2 y_j^n+y_{j-1}^n}{\Delta x^2}-b_j \frac{u_j^{n+1}-u_j^{n-1}}{2 \Delta t}=0, \\ u_0^n=u_{N+1}^n=0, \\ y_0^n=y_{N+1}^n=0. & \end{cases}
\end{equation} 
According to the initial conditions, for $j=1, \ldots, N$ 
$$u_j^0=u_0(x_j)\quad \text{ and }\quad y_j^0=y_0(x_j).$$
As in \cite{gerbi2021numerical} we can use the second initial conditions  to find the values of $u$ and $y$ at time $t^1 = \Delta t$ by employing a "ghost" time-boundary (i.e. $t^{-1}=-\Delta t$ ) and the second-order centered difference formula for $j=1, \ldots, N$ :
$$
u_1\left(x_j\right)=\left.\frac{\partial u}{\partial t}\right|_{x_j, 0}=\frac{u_j^1-u_j^{-1}}{2 \Delta t}+O\left(\Delta t^2\right) .
$$

Thus we have for $j=1, \ldots, N$:
$$
u_j^{-1}=u_j^1-2 \Delta t u_1\left(x_j\right).
$$

We use the same discrete form of the initial conditions for $y$ to obtain, for $j=1, \ldots, N$, that:
$$
y_j^{-1}=y_j^1-2 \Delta t y_1\left(x_j\right).
$$

Setting $n=0$, in the numerical scheme \eqref{prbdescrr}, the two preceding equalities permit us to compute $\left(u_j^1, y_j^1\right)_{j=0, N}$ as follows  
$$u_j^1=(1-\lambda)u_j^0+\Delta t u_1(x_j)+\frac{\lambda}{2}(u_{j+1}^0+u_{j-1}^0)-b_j \frac{\Delta t^2}{2}y_1(x_j)-a_j\frac{\mu_1\Delta t^2}{2}u_1(x_j)-a_j\frac{\mu_2\Delta t^2}{2}f_0(x_j,-I\Delta t)$$
and 
$$y_j^1=(1-\lambda)y_j^0+\Delta t y_1(x_j)+\frac{\lambda}{2}(y_{j+1}^0+y_{j-1}^0)+b_j \frac{\Delta t^2}{2}u_1(x_j),$$
where $\lambda=\frac{\Delta t^2}{\Delta x^2}.$

Finally, the solution $(u, y)$ can be computed at any time $t^n$. Indeed, using the same argument as in \cite{gerbi2021numerical}  we have,
for $n=1,\ldots, I$ 
$$\begin{aligned}
	&\begin{aligned}
		u_j^{n+1}=(1- \lambda) \alpha_j u_j^n & +\lambda \beta_j\left(u_{j+1}^n+u_{j-1}^n\right)+\gamma_j u_j^{n-1}-(1-\lambda) \varrho_j y_j^n \\
		& -\lambda \zeta_j\left(y_{j+1}^n+y_{j-1}^n\right)+\kappa_j y_j^{n-1}-r_j  f_0(x_j,(n-I)\Delta t)
	\end{aligned}\\
	&\begin{aligned}
		y_j^{n+1}=(1-\lambda) \widetilde{\alpha}_j y_j^n & +\lambda \widetilde{\beta}_j\left(y_{j+1}^n+y_{j-1}^n\right)+\widetilde{\gamma}_j y_j^{n-1}+(1-\lambda) \widetilde{\varrho}_j u_j^n \\
		& +\lambda \tilde{\xi}_j\left(u_{j+1}^n+u_{j-1}^n\right)+\widetilde{\kappa}_j u_j^{n-1}-\widetilde{r}_j f_0(x_j,(n-I)\Delta t)
	\end{aligned}
\end{aligned}$$ 
and for $n>I$ 
 $$\begin{aligned}
 &\begin{aligned}
 u_j^{n+1}=(1- \lambda) \alpha_j u_j^n & +\lambda \beta_j\left(u_{j+1}^n+u_{j-1}^n\right)+\gamma_j u_j^{n-1}-(1-\lambda) \varrho_j y_j^n \\
 & -\lambda \zeta_j\left(y_{j+1}^n+y_{j-1}^n\right)+\kappa_j y_j^{n-1}-r_j  \left(\frac{u_j^{n-I+1}-u_j^{n-I-1}}{2\Delta t}\right)
 \end{aligned}\\
 &\begin{aligned}
 y_j^{n+1}=(1-\lambda) \widetilde{\alpha}_j y_j^n & +\lambda \widetilde{\beta}_j\left(y_{j+1}^n+y_{j-1}^n\right)+\widetilde{\gamma}_j y_j^{n-1}+(1-\lambda) \widetilde{\varrho}_j u_j^n \\
 & +\lambda \tilde{\zeta}_j\left(u_{j+1}^n+u_{j-1}^n\right)+\widetilde{\kappa}_j u_j^{n-1}-\widetilde{r}_j \left(\frac{u_j^{n-I+1}-u_j^{n-I-1}}{2\Delta t}\right)
 \end{aligned}
 \end{aligned}$$ 
 where we have set:
 $$\begin{aligned}
 \alpha_j  =\frac{2}{c_j}, && \beta_j  =\frac{1}{c_j}, 
 &&\gamma_j  =\frac{c_j-2}{c_j}, && \varrho_j  =\frac{b_j \Delta t}{c_j}, &&
 \zeta_j  =\frac{b_j \Delta t}{2c_j},\\ 
 \kappa_j  =\frac{\mu_1a_j}{c_j},&& r_j=\frac{\mu_2a_j \Delta t^2}{c_j},\\
 \widetilde{\alpha}_j  =2-\frac{(b_j \Delta t)^2}{2c_j}, && \widetilde{\beta}_j  =1-\frac{(b_j \Delta t)^2}{4c_j}, 
 &&\widetilde{\gamma}_j  =\frac{(b_j \Delta t)^2}{2c_j}-1, && \widetilde{\varrho}_j  =\frac{b_j \Delta t}{c_j}, &&
 \widetilde{\zeta}_j  =\frac{b_j \Delta t}{2c_j},\\
  \widetilde{\kappa}_j  =\frac{b_j\Delta t}{2}(\frac{c_j-2}{c_j}-1),&&\widetilde{r}_j= \frac{\mu_2a_jb_j \Delta t^3}{2c_j},
 \end{aligned}$$
 where $c_j=1+\frac{\mu_1a_j}{2} \Delta t+\left(\frac{b_j \Delta t}{2}\right)^2.$
 \subsection{Numerical approximation of the continuous energy}
We begin by discretizing  the four leading terms of the energy \eqref{energysystem} by employing the first procedure used in \cite{gerbi2021numerical} to obtain:
 \begin{itemize}
 	\item The discrete kinetic   energy for $u$ as: $E_{k, u}^n=\frac{1}{2} \sum_{j=1}^N\left(\frac{u_j^{n+1}-u_j^n}{\Delta t}\right)^2.$
 	\item  The discrete potential energy for $u$ as: $E_{p, u}^n=\frac{1}{2} \sum_{j=0}^N\left(\frac{u_{j+1}^n-u_j^n}{\Delta x}\right)\left(\frac{u_{j+1}^{n+1}-u_j^{n+1}}{\Delta x}\right).$
 	\item  The discrete kinetic energy for $y$ as: $E_{k, y}^n=\frac{1}{2} \sum_{j=1}^N\left(\frac{y_j^{n+1}-y_j^n}{\Delta t}\right)^2.$
 	\item The discrete potential energy for $u$ as: $E_{y, u}^n=\frac{1}{2} \sum_{j=0}^N\left(\frac{y_{j+1}^n-y_j^n}{\Delta x}\right)\left(\frac{y_{j+1}^{n+1}-y_j^{n+1}}{\Delta x}\right).$
 \end{itemize}
 We approximate the last term in the definition of the energy \eqref{energysystem} as follows:
 $$E_{d,u}^n=\begin{cases}
 \frac{\xi\tau^{-1}}{2}\sum_{j=1}^Na_j\left(\sum_{k=n+1}^{I}\Delta t f_0^2(x_j,(k-I)\Delta t)+\sum_{k=1}^{n}\Delta t\left(\frac{u_j^{k+1}-u_j^{k-1}}{2\Delta t}\right)^2\right),& \text{ if } n<I,\\
 \frac{\xi\tau^{-1}}{2}\sum_{j=1}^Na_j\Delta t\sum_{k=n-I+1}^{n}\left(\frac{u_j^{k+1}-u_j^{k-1}}{2\Delta t}\right)^2,& \text{ if } n\geq I.
 \end{cases}
 $$
 The total discrete energy is then defined as
 $$E^{n}=E_{k,u}^{n}+E_{p,u}^{n}+E_{k,y}^{n}+E_{p,y}^{n}+E_{d,u}^{n}.$$
 
 Let us prove now that this definition of the energy fulfills the  dissipaing property stated above. To this end, for $n\leq I,$
 we multiply the first equation of \eqref{prbdescrr} by $\left(u_j^{n+1}-u_j^{n-1}\right)$ and we sum up over $j=1, \ldots, N$.
 We obtain:
 $$
 \begin{aligned}
 & \sum_{j=1}^N \frac{u_j^{n+1}-2 u_j^n+u_j^{n-1}}{\Delta t^2}\left(u_j^{n+1}-u_j^{n-1}\right)- \sum_{j=1}^N \frac{u_{j+1}^n-2 u_j^n+u_{j-1}^n}{\Delta x^2}\left(u_j^{n+1}-u_j^{n-1}\right) \\
 & +\sum_{j=1}^N b_j \frac{y_j^{n+1}-y_j^{n-1}}{2 \Delta t}\left(u_j^{n+1}-u_j^{n-1}\right)+\mu_1\sum_{j=1}^N a_j \frac{\left(u_j^{n+1}-u_j^{n-1}\right)^2}{2 \Delta t}\\
 &+\mu_2\sum_{j=1}^N a_j f_0(x_j,(n-I)\Delta t)\left(u_j^{n+1}-u_j^{n-1}\right)=0.
 \end{aligned}
 $$
 For the three first leading terms, we reason by the same technique used in \cite{gerbi2021numerical} to establish that
 \begin{equation}
 \label{dissipenergyu}
 \begin{aligned}
&2\left(E_{k,u}^{n}+E_{p,u}^{n}-E_{k,u}^{n-1}-E_{p,u}^{n-1}\right)+2\Delta t \mu_1 \sum_{j=1}^N a_j \left(\frac{u_j^{n+1}-u_j^{n-1}}{2 \Delta t}\right)^2\\
&+2\Delta t \mu_2 \sum_{j=1}^N a_j f_0(x_j,(n-I)\Delta t)\left(\frac{u_j^{n+1}-u_j^{n-1}}{2 \Delta t}\right)+\sum_{j=1}^N b_j \frac{y_j^{n+1}-y_j^{n-1}}{2 \Delta t}\left(u_j^{n+1}-u_j^{n-1}\right)=0.
 \end{aligned}
 \end{equation}
 Similarly, by multiplying the second equation of \eqref{prbdescrr} by $\left(y_j^{n+1}-y_j^{n-1}\right)$, and using the same
 algebraic tricks, we arrive at 
\begin{equation}
\label{dissipenergyy}
 2\left(E_{k, y}^n+E_{p, y}^n-E_{k, y}^{n-1}-E_{p, y}^{n-1}\right)-\sum_{j=1}^N b_j \frac{u_j^{n+1}-u_j^{n-1}}{2 \Delta t}\left(y_j^{n+1}-y_j^{n-1}\right)=0.
\end{equation}
Now, we devote our interest to the last term of the energy \eqref{energysystem}. For $n<I$, we obtain 
$$
\begin{aligned}
E^{n}_{d,u}-E^{n-1}_{d,u}=& \frac{\xi\tau^{-1}}{2}\sum_{j=1}^Na_j\left(\sum_{k=n+1}^{I}\Delta t f_0^2(x_j,(k-I)\Delta t)+\sum_{k=1}^{n}\Delta t\left(\frac{u_j^{k+1}-u_j^{k-1}}{2\Delta t}\right)^2\right)\\
&-\frac{\xi\tau^{-1}}{2}\sum_{j=1}^Na_j\left(\sum_{k=n}^{I}\Delta t f_0^2(x_j,(k-I)\Delta t)+\sum_{k=1}^{n-1}\Delta t\left(\frac{u_j^{k+1}-u_j^{k-1}}{2\Delta t}\right)^2\right)\\
=&\frac{\xi\tau^{-1}}{2}\sum_{j=1}^Na_j\Delta t \left(\frac{u_j^{n+1}-u_j^{n-1}}{2\Delta t}\right)^2-\frac{\xi\tau^{-1}}{2}\sum_{j=1}^Na_j\Delta t f_0^2(x_j,(n-I)\Delta t).
\end{aligned}
$$
For $n=I,$ we have 
$$\begin{aligned}
E^{I}_{d,u}-E^{I-1}_{d,u}=&  \frac{\xi\tau^{-1}}{2}\sum_{j=1}^Na_j\Delta t\sum_{k=1}^{I}\left(\frac{u_j^{k+1}-u_j^{k-1}}{2\Delta t}\right)^2\\
&-\frac{\xi\tau^{-1}}{2}\sum_{j=1}^Na_j\left(\sum_{k=I}^{I}\Delta t f_0^2(x_j,(k-I)\Delta t)+\sum_{k=1}^{I-1}\Delta t\left(\frac{u_j^{k+1}-u_j^{k-1}}{2\Delta t}\right)^2\right)\\
=&\frac{\xi\tau^{-1}}{2}\sum_{j=1}^Na_j\Delta t \left(\frac{u_j^{I+1}-u_j^{I-1}}{2\Delta t}\right)^2-\frac{\xi\tau^{-1}}{2}\sum_{j=1}^Na_j\Delta t f_0^2(x_j,0).
\end{aligned}$$ 
Then, for $n\leq I$, we have   
\begin{equation}
\label{dissipenergydelay}
\begin{aligned}
E^{n}_{d,u}-E^{n-1}_{d,u}=\frac{\xi\tau^{-1}}{2}\sum_{j=1}^Na_j\Delta t \left(\frac{u_j^{n+1}-u_j^{n-1}}{2\Delta t}\right)^2-\frac{\xi\tau^{-1}}{2}\sum_{j=1}^Na_j\Delta t f_0^2(x_j,(n-I)\Delta t).
\end{aligned}
\end{equation}
Now, by combining \eqref{dissipenergyu}, \eqref{dissipenergyy} and \eqref{dissipenergydelay}, it holds that  
$$\begin{aligned}
E^{n}-E^{n-1}=& -\Delta t \mu_1 \sum_{j=1}^N a_j \left(\frac{u_j^{n+1}-u_j^{n-1}}{2 \Delta t}\right)^2-\Delta t \mu_2 \sum_{j=1}^N a_j f_0(x_j,(n-I)\Delta t)\left(\frac{u_j^{n+1}-u_j^{n-1}}{2 \Delta t}\right)\\
&+\frac{\xi\tau^{-1}}{2}\sum_{j=1}^Na_j\Delta t \left(\frac{u_j^{n+1}-u_j^{n-1}}{2\Delta t}\right)^2-\frac{\xi\tau^{-1}}{2}\sum_{j=1}^Na_j\Delta t f_0^2(x_j,(n-I)\Delta t)\\
\leq& \Delta t \left(-\mu_1+\frac{\mu_2}{2}+\frac{\xi\tau^{-1}}{2}\right) \sum_{j=1}^N a_j \left(\frac{u_j^{n+1}-u_j^{n-1}}{2 \Delta t}\right)^2\\
&+\Delta t \left(\frac{\mu_2}{2}-\frac{\xi\tau^{-1}}{2}\right)\sum_{j=1}^Na_j f_0^2(x_j,(n-I)\Delta t).
\end{aligned}$$
Under the two conditions \eqref{conditionmu} and \eqref{conditionxi}, the dissipating property of the energy \eqref{energysystem}, in the case $n\leq I$, holds.

 We consider now the case where $n>I$. We multiply the first equation of \eqref{prbdescr} by $\left(u_j^{n+1}-u_j^{n-1}\right)$, sum up over $j=1, \ldots, N$ and multiply the second equation of \eqref{prbdescr} by $\left(y_j^{n+1}-y_j^{n-1}\right)$, and sum up over $j=1, \ldots, N$. Consequently, similarly to the case where $n\leq I$, we obtain that 
 \begin{equation}
 \label{dissipenergyuy}
 \begin{aligned}
  &2\left(E_{k, u}^n+E_{p, u}^n+E_{k, y}^n+E_{p, y}^n-E_{k, u}^{n-1}-E_{p, u}^{n-1}-E_{k, y}^{n-1}-E_{k, y}^{n-1}\right)=-2\Delta t \mu_1 \sum_{j=1}^N a_j \left(\frac{u_j^{n+1}-u_j^{n-1}}{2 \Delta t}\right)^2\\
  &-2\Delta t \mu_2 \sum_{j=1}^N a_j \left(\frac{u_j^{n-I+1}-u_j^{n-I-1}}{2 \Delta t}\right)\left(\frac{u_j^{n+1}-u_j^{n-1}}{2 \Delta t}\right)
   \end{aligned}
 \end{equation}
For the last term corresponding to the energy, we obtain
\begin{equation}
\label{dissipenergydelay2}
\begin{aligned}
E^{n}_{d,u}-E^{n-1}_{d,u}=& \frac{\xi\tau^{-1}}{2}\sum_{j=1}^Na_j\Delta t\sum_{k=n-I+1}^{n}\left(\frac{u_j^{k+1}-u_j^{k-1}}{2\Delta t}\right)^2-\frac{\xi\tau^{-1}}{2}\sum_{j=1}^Na_j\Delta t\sum_{k=n-I}^{n-1}\left(\frac{u_j^{k+1}-u_j^{k-1}}{2\Delta t}\right)^2\\
=&\frac{\xi\tau^{-1}}{2}\sum_{j=1}^Na_j\Delta t \left(\frac{u_j^{n+1}-u_j^{n-1}}{2\Delta t}\right)^2-\frac{\xi\tau^{-1}}{2}\sum_{j=1}^Na_j\Delta t \left(\frac{u_j^{n-I+1}-u_j^{n-I-1}}{2\Delta t}\right)^2.
\end{aligned}
\end{equation}
Now, combining \eqref{dissipenergyuy} and \eqref{dissipenergydelay2}, it follows that 
$$\begin{aligned}
E^{n}-E^{n-1}=& -\Delta t \mu_1 \sum_{j=1}^N a_j \left(\frac{u_j^{n+1}-u_j^{n-1}}{2 \Delta t}\right)^2-\Delta t \mu_2 \sum_{j=1}^N a_j \left(\frac{u_j^{n-I+1}-u_j^{n-I-1}}{2\Delta t}\right)\left(\frac{u_j^{n+1}-u_j^{n-1}}{2 \Delta t}\right)\\
&+\frac{\xi\tau^{-1}}{2}\sum_{j=1}^Na_j\Delta t \left(\frac{u_j^{n+1}-u_j^{n-1}}{2\Delta t}\right)^2-\frac{\xi\tau^{-1}}{2}\sum_{j=1}^Na_j\Delta t \left(\frac{u_j^{n-I+1}-u_j^{n-I-1}}{2\Delta t}\right)^2\\
\leq& \Delta t \left(-\mu_1+\frac{\mu_2}{2}+\frac{\xi\tau^{-1}}{2}\right) \sum_{j=1}^N a_j \left(\frac{u_j^{n+1}-u_j^{n-1}}{2 \Delta t}\right)^2\\
&+\Delta t \left(\frac{\mu_2}{2}-\frac{\xi\tau^{-1}}{2}\right)\sum_{j=1}^Na_j \left(\frac{u_j^{n-I+1}-u_j^{n-I-1}}{2\Delta t}\right)^2.
\end{aligned}$$
Thus, under the two conditions \eqref{conditionmu} and \eqref{conditionxi}, we have the dissipating property of the energy \eqref{energysystem} holds in the case where $n> I.$
\section{Validation of the theoretical results by numerical simulations}
\label{section5}
Based on the results of the preceding section, we conduct various numerical simulations in the aim of establishing the theoretical results. To this end, throughout this section, we set 
$$
u_0(x)=x(x-1),\quad u_1(x)=x(x-1),\quad y_0(x)=-x(x-1),\quad y_1(x)=-x(x-1),
$$
and 
$$f_0(x,t)=x(x-1)e^{-t}.$$

The mesh size is chosen as $N=10$ so that $\Delta x=0.1.$

In \cite{gerbi2021numerical} the authors described in detail the effect of Condition \eqref{conditionbdamping} in the case of strongly coupled wave equations without time delay. That being said, we devote our interest to illustrating the effect of the delayed term. To this end, we set
$$b=\mathrm{1}_{[0.1,0.2]}(x) \text{ and } a=\mathrm{1}_{[0.1,0.2] \cup[0.8,0.9]}(x).$$

We also choose $\tau=2$ and discretize the interval $[0,\tau]$ by taking a step size  $\Delta t=0.01$ into  $I=200$  sub-intervals. On the other hand, the final time $T$ is chosen as $T = 500$  and $\xi=\mu_1\tau$ so that Condition \eqref{conditionxi} is satisfied.

We recall that in \cite{nicaise2006stability}, the author proved that the single wave equation with time delay \eqref{singleproblem} is exponentially stable if and only if Condition \eqref{conditionmu} is satisfied. In order to validate the different theoretical results, we will treat two cases. The first case is when Condition \eqref{conditionmu} is satisfied. In this case, we illustrate the influence of the value of $\mu_2$ corresponding to the delayed damping term. In the second case, we will illustrate that unstabillity occurs when $\mu_2\geq\mu_1.$

The interpretation of the obtained numerical results is distinguished according to the following cases:
\begin{enumerate}
	\item \textbf{The case $\mu_2< \mu_1.$ Exponential stability.} 
In the case where $\mu_2< \mu_1$, Fig. \ref{fig: solution in the stable case} illustrates the dynamics of $u$ and $y$ versus time $t$ for large time, where we have chosen $\mu_1 =1$  and $\mu_2=0.5$. This choice verifies the assumption $\mu_2< \mu_1$. It can be observed that  the dynamics exponentially converge to the equilibrium point. Figure \ref{fig: energy in the stable case} shows the evolution of the energy $E$, as defined by \eqref{energysystem}, with respect to time $t$ for different values of $\mu_2$. It can be seen that the energy converges exponentially faster as the values of $\mu_2$ decrease from $\mu_2=0.75$ to $\mu_2 =0$. We notice that the case $\mu_2=0$ is the case when our problem is without delay. In this case, the energy exhibits the fastest exponential decay rate. These results are in line with the theoretical findings established in Theorem \ref{Theoremexpstability}.

 \item  \textbf{The case $\mu_2\geq \mu_1.$ Unstability results.} In the case where $\mu_2\geq \mu_1$,  Fig. \ref{fig: solution in the unstable case}
 presents the dynamics of $y$ and $u$ and Fig. \ref{fig: energy in the unstable case} illustrates the evolution of the energy $E$ with respect to time $t$ when $\mu_1=1$ and $\mu_2=1.2$. These figures indicate that the dynamics diverge. Therefore, we infer that the presence of a time delay $\tau=2$ with a dominant coefficient $\mu_2\geq \mu_1$ destabilizes the dynamics when $\mu_2< \mu_1$. This result is in accordance with the theoretical results given in \cite{nicaise2006stability} for the single wave equation \eqref{singleproblem}.
\end{enumerate}
\begin{figure}[H]
	\centering
	\includegraphics[width=0.88\linewidth]{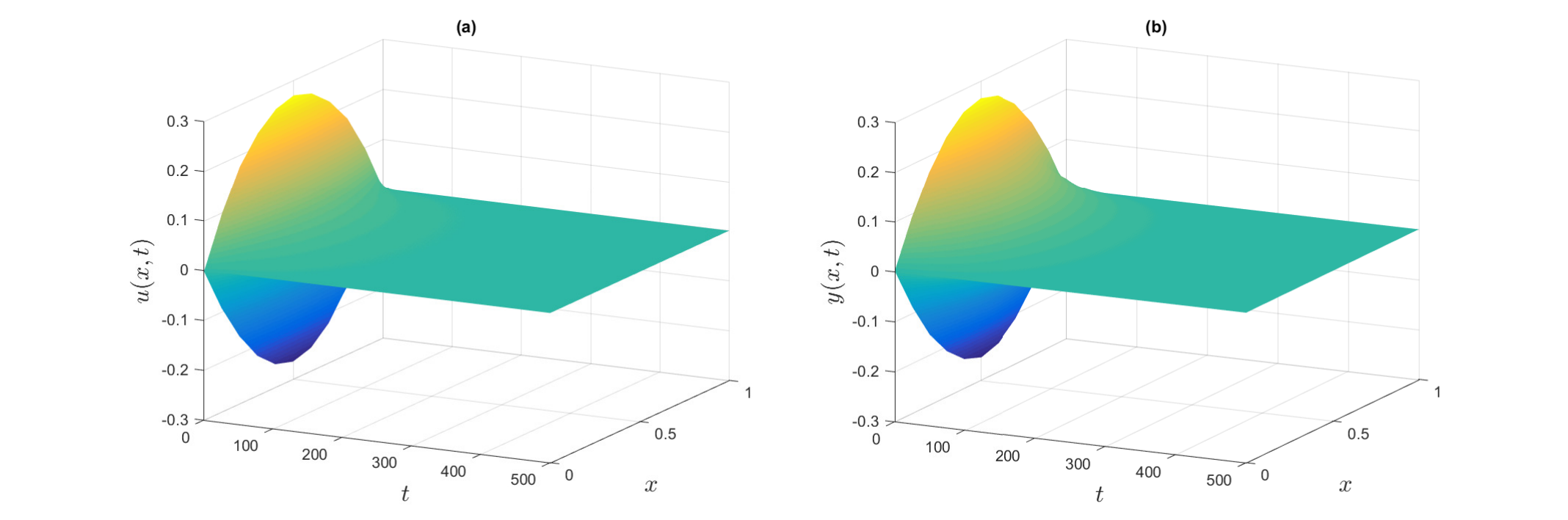}
	\caption{A 3-d landscape of the dynamics of the coupled  wave equation with time  delay $\tau=2$, when $\mu_1=1$ and  $\mu_2=0.5;$ (a) the component $u;$ (b) the component $y.$}
	\label{fig: solution in the stable case}
\end{figure}
\begin{figure}[H]
	\centering
	\includegraphics[width=0.88\linewidth]{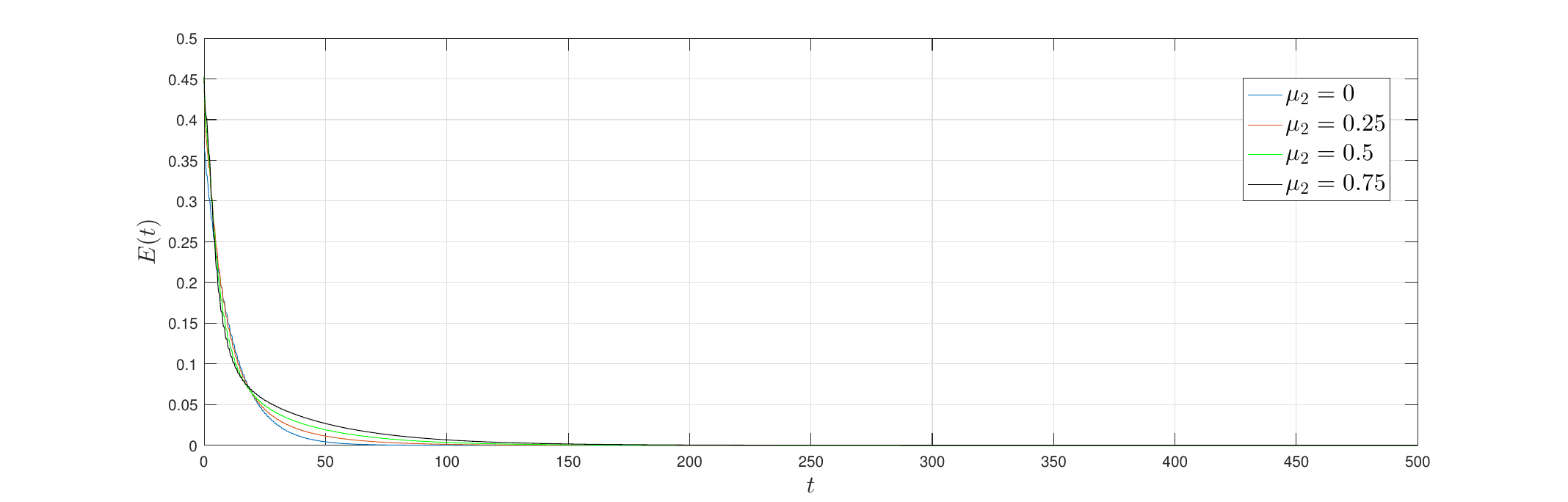}
	\caption{The energy, $E(t)$, of the  coupled waves equations  with time delay $\tau=2$,  when $\mu_1=1$ versus time
		for various values of $\mu_2$.}
	\label{fig: energy in the stable case}
\end{figure}
\begin{figure}[H]
	\centering
	\includegraphics[width=0.88\linewidth]{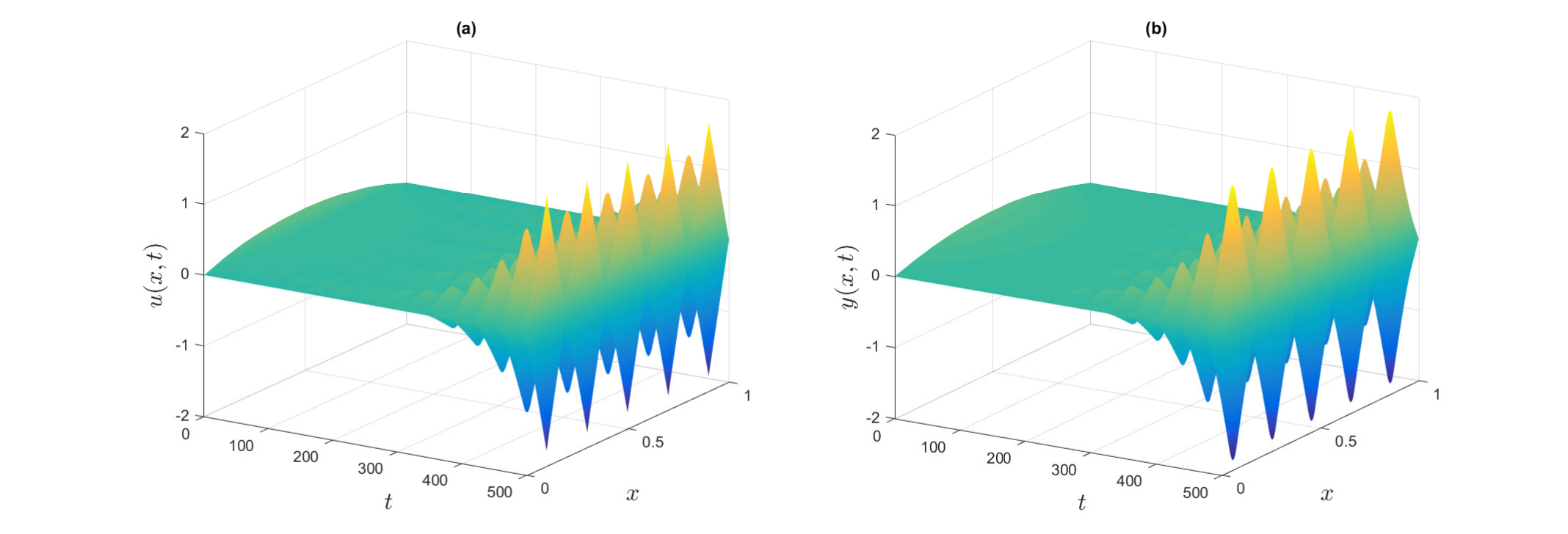}
	\caption{A 3-d landscape of the dynamics of the coupled  wave equation with time  delay $\tau=2$, when  $\mu_1=1,$ $\mu_2=1.2$ and $\xi=0;$ (a) the component $u;$ (b) the component $y$.}
	\label{fig: solution in the unstable case}
\end{figure}
\begin{figure}[H]
	\centering
	\includegraphics[width=0.88\linewidth]{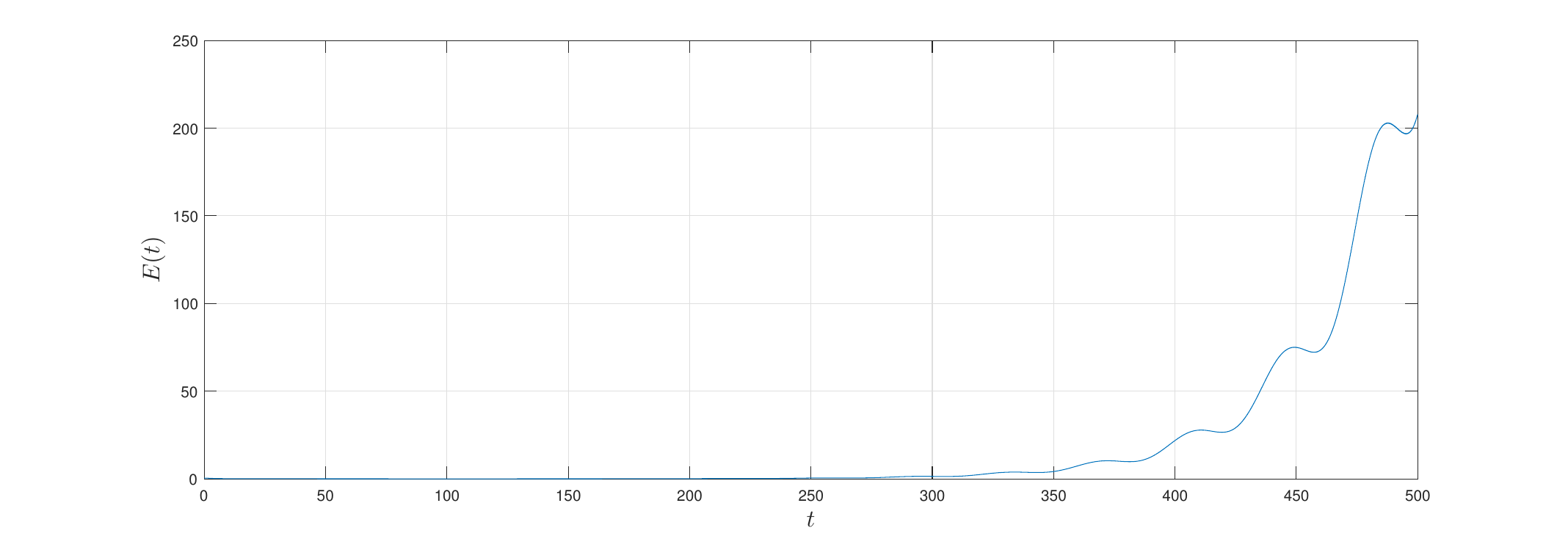}
	\caption{The energy, $E(t)$, of the  coupled waves equations  with time delay $\tau=2$, versus time
		for $\mu_1=1,$  $\mu_2=1.2$ and $\xi=0$.}
	\label{fig: energy in the unstable case}
\end{figure}
\section*{Conflicts of Interest}
The authors declare that they have no conflict of interest.



\bibliographystyle{unsrt}
\bibliography{MPS-bibliography-v}


\end{document}